\documentclass[11pt]{article}
\usepackage{amsmath}
\usepackage{amssymb,enumerate}
\usepackage{mathabx}
\usepackage{amsfonts}
\usepackage{amsthm}
\usepackage{ascmac}
\usepackage{latexsym}
\usepackage{mathrsfs}

\usepackage{cases}
\DeclareFontFamily{U}{mathc}{}
\DeclareFontShape{U}{mathc}{m}{it}%
{<->s*[1.03] mathc10}{}

\DeclareMathAlphabet{\mathscr}{U}{mathc}{m}{it}

\usepackage{enumerate}
\usepackage[numbers]{natbib}
\usepackage{mathscinet}
\usepackage{setspace}

\usepackage[pdftex]{graphicx}

\usepackage{hyperref}
\hypersetup{colorlinks=true,linkcolor=blue,
citecolor=blue}

\usepackage{floatflt}

\usepackage{mathptmx}

\usepackage[scaled]{helvet}
\usepackage{fancybox,ascmac}

\usepackage{bbm}

\usepackage[inline]{enumitem}

\usepackage{latexsym}
\newsavebox{\toy}
\savebox{\toy}{\framebox[0.65em]{\rule{0cm}{1ex}}}
\newcommand{\QED}{\usebox{\toy}\end{demo}}

\theoremstyle{theorem}
\newtheorem{thm}{Theorem}[section]
\newtheorem*{thm*}{Theorem}

\newtheorem{prop}[thm]{Proposition}
\newtheorem{lem}[thm]{Lemma}

\newtheorem{rem}[thm]{Remark}
\newtheorem{conjecture}[thm]{Conjecture}

\makeatletter
             \@addtoreset{equation}{section}
\makeatother

\newcommand{\dis}{\displaystyle}

\usepackage{comment}

\newcommand{\R}{{\mathbb{R}}}

\newcommand{{\rd}}{\R^d}

\newcommand{\e}{\varepsilon}

\newcommand{\dd}{\text{\rm d}}             

\newcommand{\bsm}{\begin{smallmatrix}}
\newcommand{\esm}{\end{smallmatrix}}

\newcommand{\MN}[1]{\textcolor{red}{#1}}

\makeatletter
\newcommand{\subsubsubsection}{\@startsection{paragraph}{4}{\z@}%
  {1.0\Cvs \@plus.5\Cdp \@minus.2\Cdp}%
  {.1\Cvs \@plus.3\Cdp}%
  {\reset@font\sffamily\normalsize}
}
\makeatother
\usepackage{makeidx}

\makeindex

\begin{document}
\author{Makoto Nakashima}
\title{An upper bound for the lower tail of the mass of balls under the  Critical $2d$ stochastic heat flow}
\date{}
\maketitle

\begin{abstract}
We study the critical two-dimensional stochastic heat flow $\mathscr{Z}_t^{\vartheta}$, recently constructed as the scaling limit of directed polymers in a random environment and as the weak limit of the solution to a mollified stochastic heat equation. Focusing on the mass of balls $\mathscr{Z}_t^{\vartheta}(B_r(0),B_r(a))$
 ($a\in \R^2$, $r>0$), we establish an upper bound on its lower tail. As a consequence, we prove the integrability of the logarithm of $\mathscr{Z}_t^{\vartheta}(B_r(0),B_r(a))$
 and its strict positivity. These results provide partial answers to open questions concerning the local behavior of $\mathscr{Z}_t^\vartheta$.\end{abstract}

\section{Introduction and main results}

The Kardar-Parisi-Zhang (KPZ) equation is a stochastic partial differential equation that descirbes the evolution of randomly growing interfaces:\begin{align}
\partial_t h=\nu \Delta h+\frac{\lambda}{2}\left|\nabla h\right|^2 +\dot{\mathcal{W}}\label{eq:KPZeq}\tag{KPZ},
\end{align}
where $\dot{\mathcal{W}}$ denotes a space-time white noise on $[0,\infty)\times\mathbb{R}^d$ \cite{KPZ86}. The equation is ill-posed due to the term $|\nabla h|^2$, which involves the square of a distribution. An approach to overcome this difficulty is  the Cole-Hopf transform, as introduced in \cite{BG97}.

In \cite{KPZ86}, the formal transformation
\begin{align*}
u(t,x)=\exp\left(\frac{\lambda}{2\nu}h(t,x)\right)
\end{align*}
was introduced, along with  the associated  stochastic heat equation with multiplicative noise\begin{align}
\partial_t u=\nu \Delta u+\frac{\lambda}{2\nu}u\dot{\mathcal{W}}.\label{eq:SHE}\tag{SHE}
\end{align} 

In \cite{BG97}, Bertini and Giacomin provided a rigorous mathematical treatment of this transformation in one spatial dimension, particularly in the context of the weakly asymmetric simple exclusion process and the solid-on-solid (SOS) process. They first considered the mollified stochastic heat equation: \begin{align}
\partial_t u^\e=\frac{1}{2} \Delta u^\e-\lambda u^\e\dot{\mathcal{W}}^\e,\label{eq:SHEe}\tag{$\text{SHE}_\e$}
\end{align}
with initial condition $u_0(x)>0$, where the mollified noise is defined by $\dot{\mathcal{W}}^\e(t,x):=\int_{\R}j_\e(x-y)\dot{\mathcal{W}}(t,y)\dd y$  for a symmetric probability density $j\in C_c^\infty(\R)$ with $j(x)=j(-x)$ and $j_{\e}(x)=\e^{-1}j\left(\frac{x}{\e}\right)$. Then, Bertini and Giacomin showed that $h^\e(t,x):=\log u^\e(t,x)$ satisfied a mollified version of the KPZ equation \begin{align}
\partial_t h^\e=\frac{1}{2}\Delta h^\e+\frac{\lambda}{2}\left(|\nabla h^\e|^2-J_\e(0)\right)+\dot{\mathcal{W}}\label{eq:KPZe}\tag{KPZ$_\e$},
\end{align}
where $J_\e(y)=\frac{1}{\e}\int j(x)j(y-x)\dd x$. 
Furthermore, they proved that $u^\e$ converged  almost surely and in $L^p$ to  $u$, the solution of \eqref{eq:SHE}, uniformly on compacts sets of $[0,\infty)\times \R$. 

It is also known \cite{Mue91,Mor14}  that if $u_0$ is nonnegative, not identically zero, and continuous, then $u(t,x)$ is strictly positive on $\R$ for all $t>0$.  Consequently, $h^\e$ converges almost surely to the process $\mathfrak{h}(t,x):=\log u(t,x)$,  known as the \textit{Cole-Hopf solution} of the KPZ equation \eqref{eq:KPZeq}. 
\begin{align}
u^\varepsilon\Rightarrow u\overset{\text{C-H trans.}}{\leadsto} h^\varepsilon\Rightarrow \mathfrak{h}.\label{eq:Cole-Hopf} \tag{Cole-Hopf}
\end{align} 
 
 In recent years, the KPZ equation has received considerable attention in the study of singular stochastic partial differential equations (SPDEs). Notably, Hairer developed the theory of \textit{regularity structures} and established the existence of a distributional solution to the KPZ equation \cite{Hai13,Hai14,FH14}. Alternative approaches include the \textit{paracontrolled calculus} by  Gubinelli-Imkeller-Perkowski \cite{GIP15}, the \textit{energy solutions} by Gon\c{c}alves-Jara \cite{GJ14}, and  the \textit{renormalization group approach} by Kupiainen\cite{Kup16}. 

\vspace{1em}
Considerable efforts have been devoted to construct solutions to \eqref{eq:KPZeq} for $d\geq 2$.
In the literature of singular SPDEs, the cases $d=1$ , $d=2$, and $d\geq 3$ are referred to as \textit{sub-critical}, \textit{critical}, and \textit{super-critical} respectively \cite{Hai14}. In the terminology of physics,  $d=1$ and $d\geq 3$  correspond to the   \textit{ultraviolet superrenormalizable} and \textit{infrared renormalizable} cases, respectively \cite{MU18}.

We now review recent progress made in the two-dimensional case. Consider the stochastic heat equation with appropriately tuned noise: \begin{align}
\partial_t u^{\vartheta,\e}=\frac{1}{2}\Delta u^{\vartheta,\e}-\beta_{\vartheta,\e} u^{\vartheta,\e}\dot{\mathcal{W}}^\e,\quad u^{\vartheta,\e}(0,x)=u_0(x),\label{eq:SHEbetae}\tag{SHE$_{\vartheta,\e}$}
\end{align} 
where the mollified noise is defined by \begin{align*}
\dot{\mathcal{W}}^\e(t,x):=\int_{\R^2}j_\e(x-y)\dot{\mathcal{W}}(t,y)\dd y
\end{align*}
with a symmetric probability density $j\in C_c^\infty(\R^2)$ satisfying $j(x)=j(-x)$, and $j_{\e}(x)=\e^{-2}j\left(\frac{x}{\e}\right)$.

 The tuning parameter $\beta_{\vartheta,\e}$ is defined by \begin{align*}
\beta_{\vartheta,\e}=
\sqrt{\frac{2\pi}{-\log \e}+\frac{\rho+o(1)}{\left(-\log \e\right)^2}},
\end{align*}
where $\rho=\pi \vartheta+C$ for $\vartheta\in \R$ and constant $C$ given by \begin{align*}
-\pi \log 4-2\pi \int_{\R^2\times\R^2}J(x)\log \frac{1}{|x-y|}J(y)\dd x\dd y+\pi \gamma,
\end{align*}
where $\gamma=\lim_{n\to\infty}\left(\sum_{k=1}^n\frac{1}{k}-\log n\right)$ is the Euler-Mascheroni constant.

It is straightforward to see that for $u_0,\psi\in L^2(\R^2)$, \begin{align*}
E\left[\int_{\R^2}u^{\vartheta,\e}(t,x)\psi(x)\dd x\right]=\int_{\R^2\times \R^2}u_0(x)p_t(x,y)\psi(y)\dd x\dd y,
\end{align*}
where \begin{align*}
p_t(x)=\frac{1}{2\pi t}\exp\left(-\frac{|x|^2}{2t}\right)
\end{align*}
is the Gaussian density function with the mean $0$ and the variance $t$, and we set $p_t(x,y)=p_t(y-x)$. Therefore, the tightness of $\int_{\R^2}u^{\vartheta,\e}(t,x)\psi(x)\dd x$ follows. Moreover, it is known that 
for $u_0,\psi\in L^2(\R^2)$, the variance of $\int_{\R^2}u^{\vartheta,\e}(t,x)\psi(x)\dd x$ converges to a nontrivial quantity of the same form as in \cite{BC98}, suggesting the existence of a nontrivial  random field. Furthermore,  Caravenna-Sun-Zygouras \cite{CSZ19b} and Gu-Quastel-Tsai \cite{GQT21} proved that  the higher moments of $\int_{\R^2}u^{\vartheta,\e}(t,x)\psi(x)\dd x$ were finite and converged. However, these moments grow too rapidly to determine the distribution on their own, and the uniqueness of the limit was established in \cite{CSZ23}  in the directed polymers setting (in the sense of finite-dimensional time distributions). In addition, Tsai \cite{Tsa24} proved the convergence of $u^{\vartheta,\e}(t,x)$ as a continuous measure-valued process,  where the space of  measures is equipped with the vague topology. Therefore, the limit $\mathscr{Z}_t^{\vartheta}(u_0,\dd x)$ may be regarded as a solution to \eqref{eq:SHE} in dimension $d=2$. Recently, several studies have investigated the  properties of $\mathscr{Z}^\vartheta$ \cite{CM24,CSZ24,LZ24,Nak25,Che25,CSZ25,GT25,CD25,CT25,GN25,BCT25,GT26}.

In \cite{CSZ25}, it was proved that $\mathscr{Z}^\vartheta$ was singular with respect  to the Lebesgue measure on $(\R^2,\mathcal{B}(\R^2))$.
Moreover, they proved that for any fixed $t>0$ and $\vartheta\in\R$\begin{align*}
\lim_{\e\to0}\frac{1}{2\pi \e^2}\mathscr{Z}_t^\vartheta(1,B_\e(x))= 0, \text{ a.e.~$x\in \R^2$, }
\end{align*} 
$P$-almost surely, where we denote by $\mathscr{Z}^\vartheta_t(u_0,A)$ the volume of a Borel set $B\in\R^2$ with respect to the measure $\mathscr{Z}^\vartheta_t(u_0,\dd x)$ and  $B_r(x)=\{y\in\R^2:|y-x|<r\}$ is the open ball centered at $x$ with the radius $r>0$. 

Thus, the Cole-Hopf transform cannot be directly applied to  $\mathscr{Z}^\vartheta$. Therefore, we need to consider further renormalization:
If one can find functions $a(\e)$ and $b(\e)$ such that \begin{align*}
\int_{\R^2}\frac{1}{b(\e)}\left(\log \mathscr{Z}^\vartheta_t(1,B_\e(x))-a(\e)\right)\psi(x)\dd x,\quad \psi\in C_c(\R^2)
\end{align*}
converges to a nontrivial random variable $\mathfrak{h}_t(\psi)$, we may regard $\mathfrak{h}_t(\psi)$ as the solution to the two-dimensional KPZ equation \eqref{eq:KPZeq}. In the subcritical regime of $u^{\vartheta,\e}$, where $\beta_{\vartheta,\e}$ is replaced by $\rho\beta_{\vartheta,\e}$ for $\rho\in (0,1)$, a similar renormalization has been considered for $u^{\vartheta,\e}(t,x)$ with $a(\e)=E[\log u^{\vartheta,\e}(t,x)]$ and $b(\e)=\beta_{\vartheta,\e}$. 
Firstly, Chatterjee and Gu established the tightness of this renormalized field \cite{CD20}. Subsequently, it was shown by Gu \cite{Gu20} for small $\rho<1$, and later  by Caravenna-Sun-Zygouras \cite{CSZ20a} for all $\rho\in (0,1)$, that the resulting field converges to the Edwards-Wilkinson equation, which is Gaussian. For general initial data, it is further known that the limiting field solves the Edwards-Wilkinson equation with an additional  linear drift term \cite{NN23}.

Thus, a natural choice of $a(\e)$ in the critical case is $E\left[\log \mathscr{Z}^\vartheta_t(1,B_\e(x))\right]$, but the integrability of the logarithm of  $\mathscr{Z}^\vartheta_t(1,B_\e(x))$ had not been established so far. 

Our first main result settles this integrability issues.

\begin{rem}
Recently \cite{GT26}, it has been proved that there exists a deterministic constants $\alpha_\e\to 0$ such that \begin{align*}
\frac{1}{\sqrt{\log \log \e^{-1}}}\left(\log \mathscr{Z}_1^\vartheta(1,p_{\e^2}(0))+\frac{1+\alpha_\e}{2}\log\log \e^{-1}\right)\Rightarrow \text{standard Gaussian}, \text{ as $\e\to0.$}
\end{align*}
\end{rem}

\begin{thm}[Upper bound of the lower tail]\label{thm:localposi}
Fix $\alpha>2$. Then, for each $\vartheta\in \mathbb{R}$, $t>0$, $a\in \R^2$ and $r>0$, there exists a constant  $C_{\vartheta,\alpha,a,r,t}>0$ such that for all sufficiently large  $x>1$,
\begin{align*}
P\left(\log \mathscr{Z}^\vartheta_t(B_r(0),B_r(a))<-x\right)\leq  \exp\left(-C_{\vartheta,\alpha,a,r,t}{x^\frac{2}{\alpha+1}}\right) .
\end{align*} 
\end{thm}

As an immediate consequence, we obtain the integrability of the logarithm:
\begin{thm}[Integrability of logarithm]
For any $p>1$, $\log \mathscr{Z}_t^{\vartheta}(B_r(0),B_r(a))$ has a finite $p$-th moment for every $\vartheta\in \mathbb{R}$,  $t>0$, $a\in \R^2$ and $r>0$.
\end{thm}

We also address an open question posed in \cite[Section 11, Open Question 2]{CSZ24}, concerning  the strict local positivity of $ \mathscr{Z}^\vartheta (1,B_\e(x))$: Let  $u_0\in C_{c,+}(\R^2)$ not be a constant function $0$, where $C_{c,+}(\R^2)$ denotes the set of nonnegative continuous functions with compact support. It had remained open  whether  the following holds:\begin{align*}
\mathscr{Z}^\vartheta_t(u_0,B_r(x))>0 \text{ for any $t>0$, $r>0$ and $x\in \R^2$, $P$-almost surely,}
\end{align*}

Theorem \ref{thm:localposi} yeilds a partial affirmative answer:
\begin{thm}[Strict local positivity]\label{thm:localposi2}
Let  $u_0\in C_{c,+}(\R^2)$ be such that  $u_0(x_0)>0$ for some $x_0\in\R^2$. Then,  for each fixed $t>0$,
\begin{align*}
\mathscr{Z}^\vartheta_t(u_0,B_r(x))>0 
\end{align*} 
\text{ for $r>0$ and $x\in \R^2$, $P$-almost surely.}
\end{thm}

\begin{rem}
Recently, Clark and Tsai \cite{CT25} independently proved strict local positivity of $\mathscr{Z}^\vartheta_t(u_0,B_r(x))$ via a different method.
\end{rem}

Therefore, we formulate the problem of the  two-dimensional KPZ equation \eqref{eq:KPZeq} as follows:
\begin{conjecture}
There exists a non-random function $b(\e)>0$ with $b(\e)\to \infty$ as $\e\to 0$ such that  for any $\psi\in C_c(\R^2)$\begin{align*}
\int_{\R^2}\frac{1}{b(\e)}\left(\log \mathscr{Z}^\vartheta_t(1,B_\e(x))-E[\log \mathscr{Z}^\vartheta_t(1,B_\e(x))]\right)\psi(x)\dd x
\end{align*}
converges weakly to a nontrivial random variable $\mathfrak{h}_t(\psi)$.
\end{conjecture}

\section{Preliminaries}\label{sec:prel}
In this section, we recall some properties that will be used in the proof. For more details, we refer the reader to \cite{CSZ24,CSZ25}.

\begin{prop}\label{prop:prop1}
Let $u_0, \phi\in C_{c,+}(\R^2)$ and $\vartheta\in \R$. Then, we have \begin{align*}
\mathscr{Z}^\vartheta_t(u_0,\phi)\geq0 
\end{align*} 
for all $t\geq 0$ almost surely.

Moreover, we have \begin{align*}
E\left[\mathscr{Z}^\vartheta_t(u_0,\phi)\right]=\int_{\R^2\times \R^2}u_0(x)p_t(x,y)\phi(y)\dd x\dd y.
\end{align*}
\end{prop}

To describe the variance of $\mathscr{Z}^\vartheta_t(u_0,\phi)$, we consider the derivative of the Volterra function \cite{Ape10,CM24}\begin{align*}
G_\vartheta(t):=\int_0^\infty \frac{st^{s-1}e^{(\vartheta-\gamma) s}}{\Gamma(s+1)}\dd s
\end{align*}
for $t\geq 0$. Additionally, we set \begin{align*}
G_\vartheta(t,x):=G_\vartheta(t)p_{\frac{t}{2}}(x)
\end{align*}
for $t>0$, $x\in\R^2$.

Then, we have the following.

\begin{prop}\cite{BC98}\label{prop:second}
Let $u^{\vartheta,\e}$ be the solution of  \eqref{eq:SHEbetae} with initial condition  $u_0\in L^2(\R^2)$. 

Then, for $\phi\in L^2(\R^2)$, we have\begin{align}
&\lim_{\e\to 0}\mathrm{Var}\left( \int_{\R^2}u^{\vartheta,\e}(t,x)\phi(x)\dd x\right)\\
&=4\pi\iint_{\R^2\times \R^2}\dd x\dd y\iint_{0<u<v<t}\dd u\dd v \left(P_uu_0(x)\right)^2 G_{\vartheta}(v-u,y-x)\left(P_{t-v}\phi(y)\right)^2\notag\\
&=\mathrm{Var}\left(\mathscr{Z}_t^\vartheta(u_0,\phi)\right),\label{eq:varint}
\end{align}
where the semigroup is defined by $P_t\phi(x)=\int_{\R^2}\phi(y)p_t(x,y)\dd y $.

The same identity  holds when $u_0\in C_b(\R^2)$ and $\phi\in C_c(\R^2).$

\end{prop}

We can obtain an upper bound of the second moment of the mass of shrinking balls by using Proposition \ref{prop:second}. 

\begin{prop}\label{prop:diver}
For each $\vartheta\in \R$ and $t>0$, there exists a constant $C_{\vartheta,t}\in (0,\infty)$ such that \begin{align*}
&\varlimsup_{\e\to 0}\frac{E\left[\mathscr{Z}_t^\vartheta(B_\e(0),B_\e(0))^2\right]
}{(2\pi \e^2)^4\left( \log \frac{1}{\e}\right)^2}\leq C_{\vartheta,t}.
\end{align*}
\end{prop}
\begin{proof} It is easy to verify that 
$\dis \frac{1_{B_\e(0)}(x)}{2\pi \e^2}\leq \sqrt{e}p_{\e^2}(x)$ for $\e>0$ and $x\in\R^2$. Hence, we have \begin{align}
&\frac{E\left[\mathscr{Z}_t^\vartheta(B_\e(0),B_\e(0))^2\right]
}{(2\pi \e^2)^4}\\
&\leq 4\pi e^2 \iint_{0<u<v<t}\dd u\dd v\iint_{\R^2\times \R^2}\dd x\dd yp_{u+\e^2}(x)^2G_\vartheta(v-u,y-x)p_{t-v+\e^2}(y)^2 \label{eq:varbdd1}
\end{align}
Since we have $p_t(x,y)p_t(x',y)=p_{2t}(x,x')p_{\frac{t}{2}}\left(\frac{x+x'}{2},y\right)$,  \begin{align}
&\text{the right-hand side of \eqref{eq:varbdd1}}\\
&=\iint_{0<u<v<t}\frac{ e^2}{4\pi (u+\e^2)(t-v+\e^2)} G_{\vartheta}(v-u)p_{\frac{u+\e^2+t-v+\e^2}{2}+2(v-u)}(0)\dd u\dd v\notag\\
&\leq \frac{e^2}{4\pi^2 t }\iint_{0<u<v<t}\frac{ G_{\vartheta}(v-u)}{ (u+\e^2)(t-v+\e^2)} \dd u\dd v.
\end{align}
We now estimate the last integral by splitting the domain into three regions $\{0<u<v\leq \frac{2t}{3}\}$, $\{\frac{t}{3}\leq u<v<t\}$, and $\{0<u<\frac{t}{3},\frac{2t}{3}<v<t\}$.
Recall  that $G_\vartheta(u)$ is integrable on every finite interval $(0,T)$ and bounded on $(a,b)$ for any $0<a<b<\infty$ \cite[Proposition 1.6]{CSZ19a}.
Hence, there exists a constant $C\in (0,\infty)$ (which may vary from line to  line) such that\begin{align*}
\iint_{0<u<v<\frac{2t}{3}}\frac{ G_{\vartheta}(v-u)}{ (u+\e^2)(t-v+\e^2)} \dd u\dd v&\leq C\iint_{0<u<v<\frac{2t}{3}}\dd u\dd v\frac{G_\vartheta(v-u)}{u+\e^2}\\
&\leq C\log \frac{1}{\e}\\
\iint_{\frac{t}{3}<u<v<t}\frac{ G_{\vartheta}(v-u)}{ (u+\e^2)(t-v+\e^2)} \dd u\dd v&\leq C\iint_{\frac{2t}{3}<u<v<t}\dd u\dd v\frac{G_\vartheta(v-u)}{t-v+\e^2}\\
&\leq C\log \frac{1}{\e}\\
\iint_{0<u<\frac{t}{3},\frac{2t}{3}<v<t}\frac{ G_{\vartheta}(v-u)}{ (u+\e^2)(t-v+\e^2)} \dd u\dd v&\leq C\iint_{0<u<v<\frac{t}{3}}\dd u\dd v\frac{1}{(u+\e^2)(t-v+\e^2)}\\
&\leq C\left(\log \frac{1}{\e}\right)^2.
\end{align*}
 
\end{proof}

\section{Proof of Theorem \ref{thm:localposi}}\label{sec:proofofthm1}

\subsection{Idea of the proof}

We first give a brief outline  of the proof with $a=0$ for simplicity.


To estimate the probability $P\left(\log \mathscr{Z}^\vartheta_t(B_r(0),B_r(0))<-x\right)$ ($x>0$), we begin with an elementary bound \begin{align*}
P\left(\log \mathscr{Z}^\vartheta_t(B_r(0),B_r(0))<-x\right)\leq P\left(\bigcap_{i=1}^{N^2} \left\{\mathscr{Z}^\vartheta_t(A_i,A_i)<e^{-x}\right\}\right)
\end{align*}
where $A_i$ ($i=1,\dots,N^2$) are disjoint balls in $B_r(0)$ with radius $\frac{r}{10N} $. If $\left\{\mathscr{Z}^\vartheta_t(A_i,A_i)\right\}$ were independent, then the right-hand side would be bounded by \begin{align*}
P\left( \mathscr{Z}^\vartheta_t(A_i,A_i)<e^{-x}\right)^{N^2}\leq  \exp\left(-N^2 P\left(\mathscr{Z}^\vartheta_t(A_i,A_i)\geq e^{-x}\right)\right).
\end{align*}
However, since $\left\{\mathscr{Z}^\vartheta_t(A_i,A_i)\right\}$ are not independent, this bound fails. To recover independence, we restrict the paths of Brownian motions in the Feynman-Kac formula of the solution of the mollified stochastic heat equation  $u^{\vartheta,\e}(t,x)$.
Specifically, we consider the Feynman-Kac formula of $\int_{B_r(0)} u^{\vartheta,\e}(t,x)\dd x$ for the initial condition $1_{B_r(0)}$: \begin{align*}
&\int_{B_r(0)} u^{\vartheta,\e}(t,y)\dd y\\
&\quad :=\int_{B_r(0)}\dd x \mathbf{E}_x\left[\exp\left(\beta_{\vartheta,\e}\int_0^t \dot{\mathcal{W}}^\e(s,\omega_s)\dd s-\frac{\beta_{\vartheta,\e}^2}{2}J_\e(0)t\right):\omega_t\in B_r(0)\right],
\end{align*}
which converges weakly to $\mathscr{Z}^\vartheta_t(B_r(0),B_r(0))$, where $\mathbf{P}_x$ is the law of the Brownian motion $\omega$ starting at $x$, $\mathbf{E}_x$ denotes the expectation with respect to $\mathbf{P}_x$, and  $J_\e(0)=\frac{1}{\e^2}\int j(x)j(-x)\dd x$. 

 Suppose $U_1,\dots,U_{N^2}\subset [0,t]\times \R^2$ are space-time regions  such that:\begin{itemize}
\item The distance between any two regions is greater than $\delta \frac{r}{N}$ for some $\delta>0$. 
\item The spatial slices at $s=0$ and $s=t$ coincide with $A_i$, i.e.~$ U_i|_{\{s=0\}}=U_i|_{\{s=t\}}=A_i$ ($i=1,\dots,N^2$).
\end{itemize} 
Then, consider\begin{align}
u^{\vartheta,\e}(U_i)=\int_{\R^2}\dd x \mathbf{E}_x\left[\exp\left(\beta_{\vartheta,\e}\int_0^t \dot{\mathcal{W}}^\e(s,\omega_s)\dd s-\frac{\beta_{\vartheta,\e}^2}{2}J_\e(0)t\right):\omega[0,t]\subset U_i\right],\label{eq:uU}
\end{align} 
where $\omega[0,t]:=\{(s,\omega_s):s\in [0,t]\}$ is the graph of $\omega$ on $[0,t]$.
These random variables are independent for sufficiently small $\e>0$. 

Moreover, $\{u^{\vartheta,\e_n}(U_i)\}_{i=1}^{N^2}$ has nontrivial weak limit point $\{\mathscr{Z}^{\vartheta}_t(U_i)\}_{i=1}^{N^2}$ which are also independent. 

Therefore, to estimate the lower tail probability,  it suffices  to control $P\left(\mathscr{Z}^\vartheta_t(U_i)\geq e^{-x}\right)$.

\subsection{Proof of Theorem \ref{thm:localposi}}

Fix $N\geq 1$, $r>0$, and $\alpha>2$.  We start with $a=0\in \R^2$ for simplicity.

Hereafter, $x$ may denote the integer part of $x\in \R$ when it is clear from the context.

We consider the curved tubes defined by \begin{align}
\mathrm{C}_{N}=\left\{(s,x)\in [0,t]\times \R^2:
|x|< r_{N,r}(s)
 \right\},\label{eq:CN}
\end{align}
where \begin{align}
r_{N,r}(s)=\begin{cases}
\frac{r}{20N}+(2s)^\frac{1}{\alpha}\quad &\text{for }0\leq s\leq \frac{t}{2}\\
r_{N,r}(t-s)\quad &\text{for }\frac{t}{2}\leq s\leq t
\end{cases}.\label{eq:radius}
\end{align}


Next, we  introduce a family of space-time regions $\mathrm{U}_n^{(N)}(r)$, that are obtained by shifting and folding $C_N$ such that those centers move along a piecewise linear curve. Precicely, they are defined as follows: \begin{align*}
\mathrm{U}_n^{(N)}(r):=\left\{(s,x)\in [0,t]\times \R^2: 
x\in \mathrm{D}_{n,r}(s)
\right\},
\end{align*}
where $\mathrm{D}_{n,r}(s)$  is a disc with center $o_{n,r}(s)$  and radius $r_{N,r}(s)$, given by: 
 \begin{align}
o_{n,r}(s)=\begin{cases}
\dis \left(\frac{4rn}{5N}+C_{\alpha,r,t}n^\alpha s,0\right) \quad &\text{for } 0\leq s\leq b_N\\
\dis \left(o_{n,r}(b_N)+(s-b_N),0\right)\quad &\text{for }b_N\leq s\leq \frac{t}{2}\\
\dis \left(o_{n,r}(t-s),0\right)\quad &\text{for }\frac{t}{2}\leq s\leq t
\end{cases},\label{eq:center}
\end{align}
where $b_N=\frac{1}{N^{\alpha-1}}$ and $C_{\alpha,r,t}>0$ is a  constant which depends on $\alpha$, $r$, and $t$, to  be chosen in Proposition \ref{prop:disjoint}. Here $o_{n,r}(s)$ moves along a piecewise linear path described \eqref{eq:center}, and  $r_{N,r}(s)$ is chosen as in \eqref{eq:radius} so that the cross-sectional radius expands near $s=0$.

Finally, we define  $\mathrm{U}_n^{(N,j)}(r)$ ($j=0,\dots,\frac{N}{15}$) by rotating $\mathrm{U}_n^{(N)}(r)$ anticlockwise by $\frac{20\pi}{N}j$ in space.  

\begin{rem}
The radius grows like $s^\frac{1}{\alpha}$ near $s=0$ to prevent Brownian paths from escaping the regions ``immediately".
\end{rem}

\begin{figure}[t]
\centering
\begin{minipage}[b]{0.49\columnwidth}
    \centering
    \includegraphics[width=0.9\columnwidth]{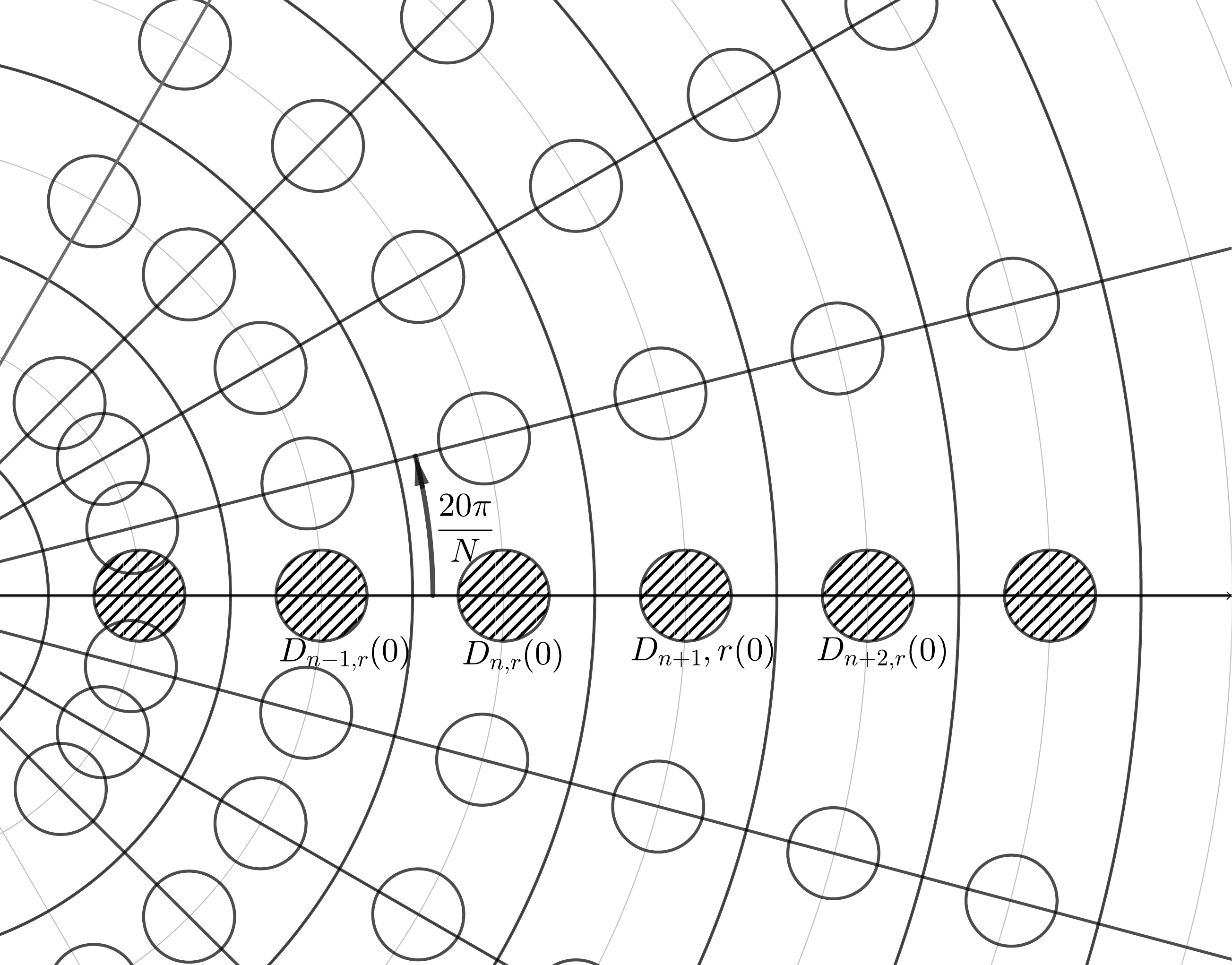}
    \caption{Image of $\{\mathrm{U}_{n}^{(N,j)}(r)|_{s=0}\}$.}
    \label{fig:a}
\end{minipage}
\begin{minipage}[b]{0.49\columnwidth}
    \centering
    \includegraphics[width=0.9\columnwidth,pagebox=cropbox,clip]{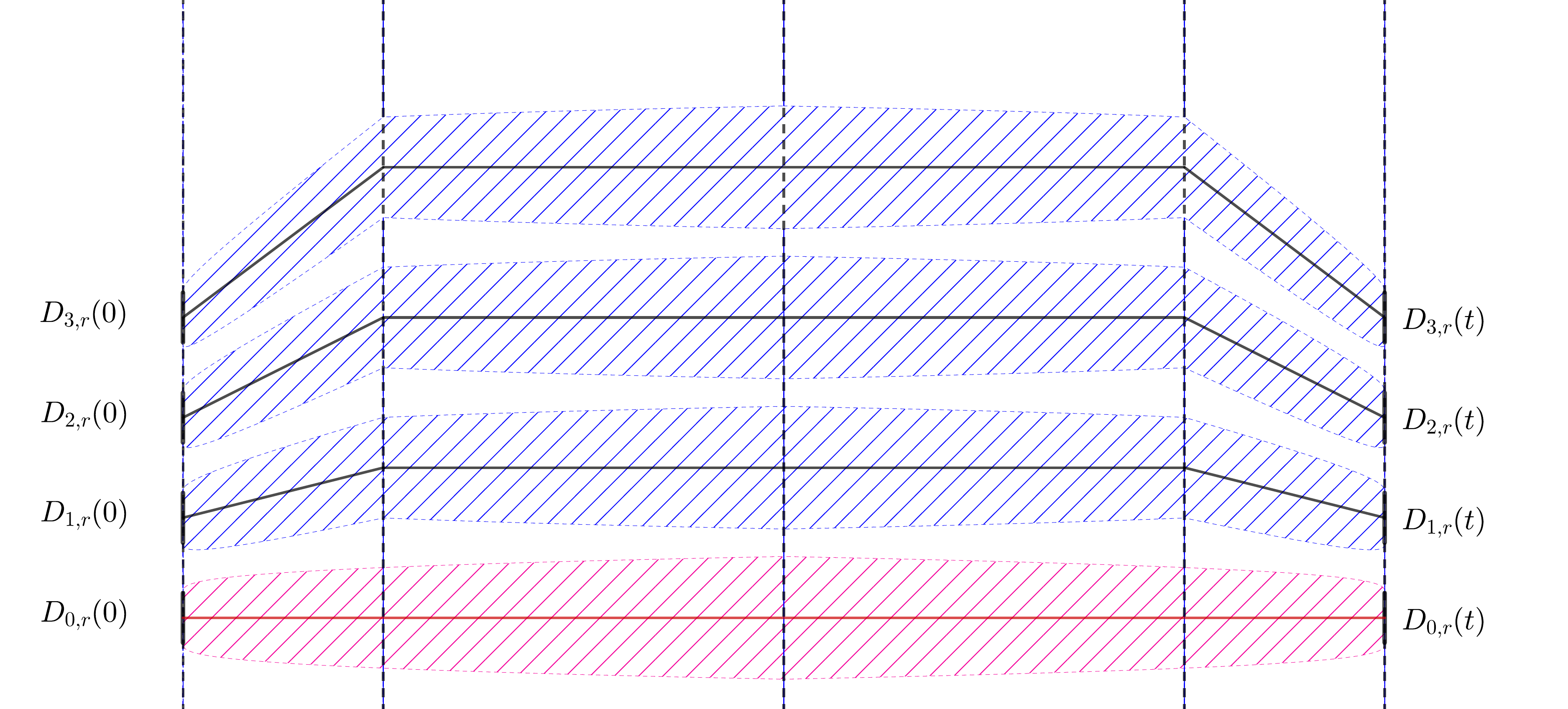}
    \caption{Image of the evolutions of $D_{n,r}$. The blue shapes are obtained from the red shape by applying a spatial shift followed by a folding transformation. The polyline represents the geometric centerline of the shape.}
    \label{fig:b}
\end{minipage}
\end{figure}

\begin{prop}\label{prop:disjoint}
For $\alpha>2 $ and $r>0$, there exists a constant $C_{\alpha,r,t}>0$ such that for large $N\geq 2$, $\{\mathrm{U}_n^{(N,j)}(r): n=\frac{N}{2},\dots,N,j=0,\dots,\frac{N}{15}\}$ are pairwise disjoint.
\end{prop}

\begin{proof}
It is enough to verify the following two statements:
\begin{enumerate}[label=(\arabic*)]
\item\label{item:xx} The maximum of the $x$-coordinate of $\mathrm{U}_n^{(N,0)}$ is strictly smaller than the minimum of the $x$ coordinate of $\mathrm{U}_{n+1}^{(N,0)}$.
\item\label{item:yy} The maximum of the $y$ coordinate of $\mathrm{U}_n^{(N,0)}$ is strictly smaller than the minimum of the $y$ coordinate of $\mathrm{U}_n^{(N,1)}$.
\end{enumerate}
We begin with \ref{item:xx}. Let  $x_{n,-}^{(0)}(s)$ and $x_{n,+}^{(0)}(s)$ denote, respectively, the minimum   and   the maximum of the $x$-th coordinate of $\mathrm{U}_n^{(N,0)}$ at time $s$.

Then, we have \begin{align}
&x_{n+1,-}^{(0)}(s)-x_{n,+}^{(0)}(s)= C_{\alpha,r,t}((n+1)^\alpha- n^\alpha) s-2(2s)^\frac{1}{\alpha} +\frac{7r}{10N} \quad \text{for }0\leq s\leq b_N\label{eq:l-l+}
\end{align}
holds. 
Using the inequality $(n+1)^\alpha-n^\alpha\geq \alpha n^{\alpha-1}$, we obtain \begin{align*}
&C_{\alpha,r,t}((n+1)^\alpha- n^\alpha) s-2(2s)^\frac{1}{\alpha} +\frac{7r}{10N}\\
&\geq C_{\alpha,r,t}\beta n^{\alpha-1} s-2(2s)^\frac{1}{\alpha} +\frac{7r}{10N} \geq -c_{\alpha}\left(C_{\alpha,r,t}\alpha\right)^\frac{\alpha}{1-\alpha}\frac{1}{N}+\frac{7r}{10N}>0
\end{align*}
for $n\geq \frac{N}{2}$, provided that  $C_{\alpha,r,t}>0$ is chosen sufficiently large. Here $c_{\alpha }>0$ is a constant depending only on $\alpha$.

Moreover, we have\begin{align*}
x_{n+1,-}^{(0)}(b_N)-x_{n,+}^{(0)}(b_N)&\geq C_{\alpha,r,t}\beta n^{\alpha-1}b_N-2(2s)^\frac{1}{\alpha}+\frac{7r}{10N}\\
&\geq CC_{\alpha,r,t}  \geq 20t^\frac{1}{\alpha}
\end{align*}
for $C_{\alpha,r,t}$ sufficiently large.

Then, it follows that  $x_{n+1,-}^{(0)}(s)>x_{n,+}^{(0)}(s)$ for $s\in [b_N,\frac{t}{2}]$ since \begin{align*}
x_{n+1,-}^{(0)}(s)-x_{n,+}^{(0)}(s)&\geq x_{n+1,-}^{(0)}(b_N)-x_{n,+}^{(0)}(b_N)-2(2s)^\frac{1}{\alpha}-2(b_N)^\frac{1}{\alpha}>0
\end{align*}
for $\frac{t}{2}\leq s\leq t$.

We now turn to \ref{item:yy}.  Let $y_{n,-}^{(1)}(s)$ denote  the minimum  of the $y$-th coordinate of $\mathrm{U}_n^{(N,1)}$ at time $s$, and let 
 $y_{n,+}^{(0)}(s)$ denote the maximum  of the $y$-th coordinate of $\mathrm{U}_n^{(N,0)}$ at time $s$.

By definition of $\mathrm{U}_n^{(N,1)}$, we have \begin{align*}
y_{n,-}^{(1)}(s)=\left(C_{\alpha,r,t}n^{\alpha} s+\frac{4r}{5N}\right)\sin \frac{20\pi}{N}-\frac{r}{20N}-(2s)^\alpha
\end{align*}
Thus, it holds that \begin{align*}
&y_{n,-}^{(1)}(s)-y_{n,+}^{(0)}(s)
\geq \widetilde{C}_{\alpha,r,t}\frac{n^{\alpha}}{N} s +\frac{15r}{N}-2(2s)^\frac{1}{\alpha}, \quad s\in [0,b_N]
\end{align*}
for $n\geq \frac{N}{2}$ for all $N\geq 2$ sufficiently large. Here $\widetilde{C}_{\alpha,r,t}$ is a constant depending on $C_{\alpha,r,t}$,  and it is increasing in $C_{\alpha,r,t}>0$.
In the same argument as \ref{item:xx}, we find that $y_{n,-}^{(1)}(s)-y_{n,+}^{(0)}(s)>0$ for $0\leq s\leq b_N$ and $y_{n,-}^{(1)}(b_N)-y_{n,+}^{(0)}(b_N)\geq 10 t^\alpha$ if  $C_{\alpha,r,t}>0$ is taken sufficiently large.

By construction, we can see that \begin{align*}
(y_{n,-}^{(1)}(s)-y_{n,+}^{(0)}(s))-(y_{n,-}^{(1)}(b_N)-y_{n,+}^{(0)}(b_N))\geq -4(2s)^\frac{1}{\alpha}\quad \text{for }b_N\leq s\leq \frac{t}{2}.
\end{align*}
Hence $y_{n,-}^{(1)}(s)-y_{n,+}^{(0)}(s)\geq 0$ for $0\leq s\leq t$.
\end{proof}

For a compact set $K\subset [0,t]\times \R^2$, we define random variables \begin{align*}
&u_t^{\vartheta,\e}(K)\\
&:=\int_{B_r(0)}\dd x \mathbf{E}_x\left[\exp\left(\beta_{\vartheta,\e}\int_0^t \dot{\mathcal{W}}^\e(s,\omega_s)\dd s-\frac{\beta_{\vartheta,\e}^2}{2}J_\e(0)t\right):\omega[0,t]\subset  K, \omega_t\in B_r(0)\right].
\end{align*}

We note that  $u_t^{\vartheta,\e}(K)$ is the same as \eqref{eq:uU} with the additional restriction on $\omega_t$, and that $u_{t}^{\vartheta,\e}([0,t]\times\R^2)$ coincides with $\int_{B_r(0)}u^{\vartheta,\e}(t,x)\dd x$. 

For $p>1$, the family $\left\{u_t^{\vartheta,\e}(\mathrm{U}_n^{(N,j)}):n=\frac{N}{2},\dots,N,j=0,\dots,\frac{N}{15}\right\}$ satisfies  \begin{align*}
E\left[\left(\sum_{n=\frac{N}{2}}^{N}\sum_{j=0}^{\frac{N}{15}}u_t^{\vartheta,\e}(\mathrm{U}_n^{(N,j)})^2\right)^p\right]&\leq E\left[\left(\sum_{n=\frac{N}{2}}^{N}\sum_{j=0}^{\frac{N}{15}}u_t^{\vartheta,\e}(\mathrm{U}_n^{(N,j)})\right)^{2p}\right]\\
&\leq E\left[u_t^{\vartheta,\e}(B_r(0),B_r(0))^{2p}\right],
\end{align*}
since $\{u_t^{\vartheta,\e}(\mathrm{U}_n^{(N,j)})\}_{j=0,\dots,\frac{N}{15},n=\frac{N}{2},\dots,N}$ are positive.

By \cite{GQT21}, the right-hand side is bounded, and hence we may extract a subsequence  $\{\e_m\}_{m\geq 1}$ such that a weak limit   \begin{align*}
&\left(\mathscr{Z}_t^\vartheta(B_r(0),B_r(0)),\left\{\mathscr{Z}_t^{\vartheta}(\mathrm{U}_n^{(N,j)}):n=\frac{N}{2},\dots,N,j=0,\dots,\frac{N}{15}\right\}\right)\\
&:=\text{w-}\lim_{m\to\infty}\left(u_t^{\vartheta,\e_m}([0,t]\times \R^2),\left\{u_t^{\vartheta,\e_m}(\mathrm{U}_n^{(N,j)}):n=\frac{N}{2},\dots,N,j=0,\dots,\frac{N}{15}\right\}\right),
\end{align*}
exists. Moreover,
\begin{align*}
\mathscr{Z}_t^\vartheta(B_r(0),B_r(0))\geq \sum_{n=\frac{N}{2}}^{N}\sum_{j=0}^{\frac{N}{15}}\mathscr{Z}_t^{\vartheta}(\mathrm{U}_n^{(N,j)}).
\end{align*}

We remark that 
$\left\{\mathscr{Z}_t^{\vartheta}(\mathrm{U}_n^{(N,j)}):n=\frac{N}{2},\dots,N,j=0,\dots,\frac{N}{15}\right\}$ are independent random variables and that $\left\{\mathscr{Z}_t^{\vartheta}(\mathrm{U}_n^{(N,j)}):j=0,\dots,\frac{N}{15}\right\}$ are identically distributed. Here, we remark that $\mathrm{U}_{n}^{(N,j)}$ and $\mathrm{U}_m^{(N,k)}$ for $n\not=m$ are not identically distributed since the shapes of the regions are different. 

Hence, we find that \begin{align}
&P\left(\log \mathscr{Z}_t^\vartheta(B_r(0),B_r(0))<-3C_{\alpha,r,t}^2N^{\alpha+1}\right)\notag\\
&\leq P\left(\bigcap_{n=\frac{N}{2}}^{N}\bigcap_{j=0}^{\frac{N}{15}}\left\{\mathscr{Z}_t^{\vartheta}(\mathrm{U}_n^{(N,j)})<e^{-3C_{\alpha,r,t}^2N^{\alpha+1}}\right\} \right)\notag\\
&=\prod_{n=\frac{N}{2}}^{N}\prod_{j=0}^{\frac{N}{15}}P\left(\mathscr{Z}_t^{\vartheta}(\mathrm{U}_n^{(N,j)})<e^{-3C_{\alpha,r,t}^2N^{\alpha+1}} \right) \notag\\
&\leq \exp\left(-\frac{N}{15}\sum_{n=\frac{N}{2}}^{N}P\left(\mathscr{Z}_t^{\vartheta}(\mathrm{U}_n^{(N,0)})\geq  e^{-3C_{\alpha,r,t}^2N^{\alpha+1}} \right)\right).\label{eq:probbdd}
\end{align}

Thus, we will complete the proof of Theorem \ref{thm:localposi} once we prove the following lemma.

\begin{lem}\label{lem:boundofexp}
For each $t>0$, $\alpha>2$, and $r>0$, there exists a constant $c_{\vartheta,\alpha,r,t}>0$ such that for $N\geq 2$ sufficiently large, we have \begin{align*}
P\left(\mathscr{Z}_t^{\vartheta}(\mathrm{U}_n^{(N,0)})\geq e^{-3C_{\alpha,r,t}N^{\alpha+1}} \right)\geq \frac{c_{\vartheta,\alpha,r,t}}{(\log N)^2}
\end{align*}
for $n\geq \frac{N}{2}$.
\end{lem}

\begin{proof}[Proof of Theorem \ref{thm:localposi}]
Combining \eqref{eq:probbdd} with Lemma \ref{lem:boundofexp},
we obtain that for all sufficiently large $N$, \begin{align*}
P\left(\log \mathscr{Z}_t^\vartheta(B_r(0),B_r(0))<-3C_{\alpha,r,t}^2N^{\alpha+1}\right)\leq \exp\left(-\frac{N^2}{30}\frac{c_{\vartheta,\alpha,r,t}}{(\log N)^2}\right).
\end{align*}
 Replacing $N$ by any real number $s>2$, and adjusting the constants  $C_{\alpha,r,t}$ and $c_{\alpha,r,t}$ accordingly, we can rewrite the above estimate in the form 
\begin{align*}
P\left(\log \mathscr{Z}_t^\vartheta(B_r(0),B_r(0))<-C_{\alpha,r,t}s^{\alpha+1}\right)\leq \exp\left(-\frac{c_{\vartheta,\alpha,r,t}s^2}{(\log s)^2}\right).
\end{align*} 
Finally, substituting $x=C_{\alpha,r,t}s^{\alpha+1}$, we obtain \begin{align*}
P\left(\log \mathscr{Z}_t^\vartheta(B_r(0),B_r(0))<-x\right)\leq \exp\left(-\frac{c_{\vartheta,\alpha,r,t}x^\frac{2}{\alpha+1}}{(\log x)^2}\right)
\end{align*}
for all $x$ sufficiently large. Since $\alpha>2$ is arbitrary, we obtain the claimed upper bound on the lower tail.
\end{proof}


\begin{proof}[Proof of Lemma \ref{lem:boundofexp}]
Let $h_n:[0,t]\times \R^2\to \R^2 $ be the path transformation  defined by \begin{align*}
h_n(s,x,y)=
\begin{cases}
\left(\frac{4nr}{5N}+x+C_{\alpha,r,t}n^\alpha s,y\right)\quad &\text{for }s\in [0,b_N]\\
\left(h_n(b_N,s,y)+(s-b_N),y\right)\quad &\text{for }s\in [b_N,\frac{t}{2}]\\
h_n(t-s,x,y)\quad &\text{for }s\in [\frac{t}{2},t].
\end{cases}
\end{align*}
Here, $h_{n}$ is a piecewise linear map that sends the straight times-axis path $\{(s,0,0):s\in[0,t]\}$ to the piecewise linear curve $\{(s,o_{n,r}(s)):s\in [0,t]\}$.

Applying the Girsanov transformation, we have \begin{align}
&u_t^{\vartheta,\e}(\mathrm{U}_n^{(N,0)})\label{eq:unun0}\\
&=\int_{D_{n,r}(0)}\dd x\mathbf{E}_x\left[\mathcal{E}_n(\omega)\exp\left(\beta_{\vartheta,\e}\int_0^t \dot{\mathcal{W}}^\e(s,h_n(s,\omega_s))\dd s-\frac{\beta_{\vartheta,\e}^2}{2}J_\e(0)t\right):\omega[0,t]\subset  \mathrm{C}_{N}\right],\notag
\end{align} 
where $C_N$ is the curved tube defined in \eqref{eq:CN} and the Radon-Nikodym deriavtive $\mathcal{E}_n(\omega)$ is given by\begin{align*}
\mathcal{E}_n(\omega)&=\exp\left(C_{\alpha,r,t}n^\alpha(\omega_{b_N}-\omega_0)-\frac{C_{\alpha,r,t}^2n^{2\alpha}b_N}{2}\right)\\
&\times \exp\left((\omega_{\frac{t}{2}}-\omega_{b_N})-\frac{\frac{t}{2}-b_N}{2}\right)\\
&\times \exp\left(-(\omega_{t-b_N}-\omega_{\frac{t}{2}})-\frac{\frac{t}{2}-b_N}{2}\right)\\
&\times \exp\left(-C_{\alpha,r,t}n^\alpha(\omega_{t}-\omega_{t-b_N})-\frac{C_{\alpha,r,t}^2n^{2\alpha}b_N}{2}\right).
\end{align*}
We remark that if  $\omega[0,t] \subset  \mathrm{C}_N$,  $\max_{0\leq s\leq t}|\omega_s|\leq 3t^\frac{1}{\alpha}+1$. Recalling $b_N=\frac{1}{N^{\alpha-1}}$, we have \begin{align*}
\mathcal{E}_n(\omega)\geq 
\exp\left(-\frac{3}{2}C_{\alpha,r,t}^2N^{\alpha+1}\right)\quad \text{for }\omega\in \mathrm{C}_n, 
\end{align*}
for $n=\frac{N}{2},\dots,N$ if  $N\geq 2$ sufficiently large.

Thus, we can see that \begin{align*}
&u_t^{\vartheta,\e}(\mathrm{U}_n^{(N,0)})\\
&\geq \exp\left(-\frac{3}{2}C_{\alpha,r,t}^2N^{\alpha+1}\right)\\
&\hspace{2.5em}\times \int_{D_{n,r}(0)}\dd x\mathbf{E}_x\left[\exp\left(\beta_{\vartheta,\e}\int_0^t \dot{\mathcal{W}}^\e(s,h_n(s,\omega_s))\dd s-\frac{\beta_{\vartheta,\e}^2}{2}J_\e(0)t\right):\omega[0,t]\subset  \mathrm{C}_{N}\right].
\end{align*}

Here, we note that \begin{align*}
&\int_{D_{n,r}(0)}\dd x\mathbf{E}_x\left[\exp\left(\beta_{\vartheta,\e}\int_0^t \dot{\mathcal{W}}^\e(s,h_n(s,\omega_s))\dd s-\frac{\beta_{\vartheta,\e}^2}{2}J_\e(0)t\right):\omega[0,t]\subset   \mathrm{C}_{N}\right]\\
&\overset{d}{=}u_t^{\vartheta,\e}(\mathrm{C}_{N}).
\end{align*}

Then, we have \begin{align*}
P\left(\mathscr{Z}_t^{\vartheta}(\mathrm{U}_n^{(N,0)})\geq  e^{-3C_{\alpha,r,t}^2N^{\beta+1}} \right)&\geq \varlimsup_{m\to \infty}P\left(u_t^{\vartheta,\e_m}(\mathrm{U}_n^{(N,0)})\geq  e^{-3C_{\alpha,r,t}^2N^{\alpha+1}} \right)\\
&\geq \varliminf_{m\to \infty}P\left(u_t^{\vartheta,\e_m}(\mathrm{C}_{N})>  e^{-2C_{\alpha,r,t}^2N^{\alpha+1}} \right)\\
&\geq P\left(\mathscr{Z}_t^{\vartheta}(\mathrm{C}_{N})>  e^{-2C_{\alpha,r,t}^2N^{\alpha+1}} \right),
\end{align*}
where $\mathscr{Z}_t^{\vartheta}(\mathrm{C}_{N})$ is a weak limit of a subsequence of $\{u_t^{\vartheta,\e_m}(\mathrm{C}_{N})\}_{m\geq 1}$, which is $L^p$-tight for any $p>1$.

Furthermore, one can see that \begin{align*}
&E\left[\mathscr{Z}_t^\vartheta (\mathrm{C}_{N})\right]=\int_{B_{\frac{r}{20N}}(0)}\mathbf{P}_x\left(\omega[0,t]\subset   \mathrm{C}_{N}\right)\dd x\\
&E\left[\mathscr{Z}_t^\vartheta (\mathrm{C}_{N})^2\right]\leq E\left[\mathscr{Z}_t^\vartheta \left(B_{\frac{r}{20N}}(0),B_{\frac{r}{20N}}(0)\right)^2\right].
\end{align*}
The law of iterated logarithms implies that there exists a constant $c_{t,\alpha}>0$ such that for any $n\geq 1$ and $x\in B_{\frac{r}{20N}}(0)$, \begin{align*}
\frac{ \mathbf{P}_x\left(|\omega_s|\leq r_{n,r}(s) \text{ for }0\leq s\leq \frac{t}{2},\omega_{\frac{t}{2}}\in \dd y \right)}{\dd y}\geq c_{t,\alpha}>0 
\end{align*} 
uniformly \text{for }$|y|\leq \frac{1}{2}t^\frac{1}{\alpha}$. Then, considering the reverse of Brownian motion, we obtain that 
 \begin{align*}
\mathbf{P}_x\left(\omega[0,t]\subset  \mathrm{C}_{N}\right)\geq c_{t,\alpha}^2\frac{r^2}{(20N)^2}
\end{align*}
uniformly for $x\in B_{\frac{r}{20N}}(0)$.

We have from \eqref{eq:unun0} and the Paley-Zygmund inequality \begin{align*}
 P\left(\mathscr{Z}_t^\vartheta (\mathrm{C}_{N})>e^{-2C_{\alpha,r,t}^2N^{\alpha+1}}\right)\geq  P\left(\mathscr{Z}_t^\vartheta (\mathrm{C}_{N})>\frac{1}{2}E\left[\mathscr{Z}_t^\vartheta (\mathrm{C}_{N})\right]\right)\geq \frac{1}{4}\frac{E\left[\mathscr{Z}_t^\vartheta (\mathrm{C}_{N})\right]^2}{E\left[\mathscr{Z}_t^\vartheta (\mathrm{C}_{N})^2\right]}.
  \end{align*}
Therefore, we obtain from Proposition \ref{prop:diver} that there exists a constant $\widetilde{C}_{\vartheta,t,\alpha}>0$ such that \begin{align*}
\frac{E\left[\mathscr{Z}_t^\vartheta (\mathrm{C}_{N})\right]^2}{E\left[\mathscr{Z}_t^\vartheta (\mathrm{C}_{N})^2\right]}\geq \frac{\widetilde{C}_{\vartheta,t,\alpha}}{\left(\log \frac{10N}{r}\right)^2}.
\end{align*}

\end{proof}

\begin{proof}[Proof of Theorem \ref{thm:localposi} for $a\not=0$]
It is enough to retake  the family of regions  $\mathrm{U}_n^{(N,j)}$ by tilting their centers with the drift $\frac{as}{t}$. 
\end{proof}

\textbf{Acknowledgemments} This work was supported by JSPS KAKENHI Grant Number JP22K03351, JP23K22399. The author thanks Prof. Nikos Zygouras and Prof. Rongfeng Sun for their useful comments.

\bibliographystyle{alpha}


\end{document}